\documentclass[12pt,reqno]{amsart}
\usepackage{amsfonts,amsmath,amssymb,amsthm,mathdots,mathtools}
\usepackage{dsfont,enumerate,float,graphicx,hyperref,makecell,pgfplots,thmtools,url}
\usepackage[margin=2.2cm]{geometry}

\newcommand\Area{\mathrm{Area}}

\newcommand\C{\mathbb{C}}

\newcommand\N{\mathbb{N}}

\newcommand\ol{\overline}

\newcommand\Op{\mathrm{Op}}

\newcommand\R{\mathbb{R}}

\newcommand\Sb{\mathbb{S}}

\newcommand\Span{\mathrm{span}}

\newcommand\Vol{\mathrm{Vol}}
\newcommand\ve{\varepsilon}

\newcommand\Z{\mathbb{Z}}

\theoremstyle{plain}
\newtheorem{thm}{Theorem}

\newtheorem{lemma}[thm]{Lemma}

\theoremstyle{remark}
\newtheorem*{defn}{\textbf{Definition}}
\newtheorem*{ex}{\textbf{Example}}
\newtheorem*{rmk}{\textbf{Remark}}

\numberwithin{equation}{section}

\allowdisplaybreaks
\title{The semiclassical measure of certain spherical harmonics on $\Sb^3$}
\author{Xiaolong Han}
\email{xiaolong.han@csun.edu}
\address{Department of Mathematics, California State University, Northridge, CA 91330, USA}

\thanks{} 

\begin{document}
\maketitle

\begin{abstract}
Bourgain \cite{B1} used the Rudin-Shapiro sequences to construct a basis of uniformly bounded holomorphic functions on the unit sphere in $\C^2$. They are also spherical harmonics (i.e., Laplacian eigenfunctions) on $\Sb^3\subset\R^4$. In this paper, we identify the semiclassical measure associated to these spherical harmonics by a measure supported on the family of Clifford tori in $\Sb^3$. In particular, these functions equidistribute on $\Sb^3$ in the semiclassical limit. 
\end{abstract}

\section{Introduction}
Bourgain \cite{B1} proved the existence of a uniformly bounded holomorphic basis on the unit sphere $\Sb_\C^2=\{(z,w)\in\C^2:|z|^2+|w|^2=1\}$ via an explicit construction using the Rudin-Shapiro sequences (c.f., Section \ref{sec:RS}): 
\begin{thm}[Uniformly bounded spherical harmonics by Bourgain]
Let $N\in\N$ and $\{\sigma_j\}_{j=0}^N$ be a Rudin-Shapiro sequence. For $k=0,...,N$, define
\begin{equation}\label{eq:P}
P_{N,k}(z,w)=\frac{1}{\sqrt{N+1}}\sum_{j=0}^N\sigma_je^{\frac{2\pi ijk}{N+1}}\frac{z^jw^{N-j}}{\left\|z^jw^{N-j}\right\|_{L^2(\Sb^3)}}.
\end{equation}
Then $\{P_{N,k}\}_{k=0}^N$ is an orthonormal basis in the space $\Span\{z^jw^{N-j}:j=0,...,N\}$, and there is an absolute constant $C>0$ such that
$$\sup_{(z,w)\in\Sb_\C^2}\left|P_{N,k}(z,w)\right|\le C\quad\text{for all }N\in\N\text{ and }k=0,...,N.$$
\end{thm}

These holomorphic functions $P_{N,k}$ are also spherical harmonics of degree $N$ on 
$$\Sb^3=\{q=(x_1,y_1,x_2,y_2)\in\R^4:|q|^2=|x_1|^2+|y_1|^2+|x_2|^2+|y_2|^2=1\},$$
that is, they are eigenfunctions of the Laplacian $\Delta_{\Sb^3}$ on $\Sb^3$ (equipped with the round metric):
\begin{equation}\label{eq:efn}
\Delta_{\Sb^3}P_{N,k}=-N(N+2)P_{N,k},
\end{equation}
where $N^{-1}\to0$ as $N\to\infty$ is understood as the semiclassical parameter. In this paper, we identify the semiclassical measure of these spherical harmonics $P_{N,k}$ in the phase space $T^*\Sb^3=\{(q,\xi):q\in\Sb^3,\xi\in T^*_q\Sb^3\}$, the cotangent bundle of $\Sb^3$. 

\begin{thm}[Semiclassical measure of spherical harmonics by Bourgain]\label{thm:sc}
Let $f\in C^\infty_0(T^*\Sb^3)$ and $\Op_N(f)$ be the semiclassical pseudo-differential operator with symbol $f$ (c.f., Section \ref{sec:sc}). Then for each $k=0,...,N$,
$$\lim_{N\to\infty}\left\langle\Op_N(f)P_{N,k},P_{N,k}\right\rangle_{L^2(\Sb^3)}=\int_0^1\int_{T_\rho}f\left(q,\xi_\rho\right)\,d\Area_\rho(q)d\rho,$$
where $\{T_\rho:0\le\rho\le1\}$ is the family of Clifford tori in $\Sb^3$ (c.f., Section \ref{sec:H}), $d\Area_\rho$ is the uniform measure on $T_\rho$ normalized so that $\Area(T_\rho)=1$, and $\xi_\rho=(0,\rho,1-\rho)\in T_q^*T_\rho$ for each $q\in T_\rho$.
\end{thm}

The semiclassical measure of a sequence of spherical harmonics describes the distribution of these functions in the phase space $T^*\Sb^3$. In particular, each semiclassical measure must be a probability measure on the cosphere bundle $S^*\Sb^3=\{(q,\xi)\in T^*\Sb^3:|\xi|_q=1\}$ and is invariant under the geodesic flow $G_t$ on $S^*\Sb^3$, see, e.g., Zworski \cite[Section 4.2]{Zw}. There are a large family of the invariant probability measures, each of which arises as the semiclassical measure of some sequence of spherical harmonics, see Jakobson-Zelditch \cite{JZ}. 

There are examples of spherical harmonics, particularly those from the standard basis and certain random constructions (see background below for further discussion), whose semiclassical measures can be explicitly identified. For example, the uniform measure on the flow-out of the cosphere space at a point, $\cup_{t\in\R}G_t(S_q^*\Sb^3)$, arises as the semiclassical measure of the zonal harmonics. On the other hand, the semiclassical measure of random spherical harmonics is the Liouville measure on $S^*\Sb^3$. 

Theorem \ref{thm:sc} identifies the semiclassical measure of the spherical harmonics $P_{N,k}$ by Bourgain as one supported on the set 
$$\left\{T_\rho\times\{\xi_\rho\}:0\le\rho\le1\right\}\subset S^*\Sb^3.$$ 
This result provides an explicit example of spherical harmonics exhibiting a unique localization pattern in the phase space $S^*\Sb^3$, specifically along the family of Clifford tori. 

The semiclassical measure of $P_{N,k}$ in $S^*\Sb^3$ readily imply their distribution on $\Sb^3$. In particular, $\Sb^3$ is foliated by the family of Clifford tori $\{T_\rho:0\le\rho\le1\}$, and the projection of the semiclassical measure of $P_{N,k}$ from $S^*\Sb^3$ onto $\Sb^3$ coincides with the Riemannian volume $d\Vol=d\Area_\rho d\rho$, implying that $P_{N,k}$ equidistribute on $\Sb^3$ as in Theorem \ref{thm:equid}:

\begin{thm}[Equidistribution of spherical harmonics by Bourgain]\label{thm:equid}
Let $f\in C^\infty(\Sb^3)$. Then
$$\lim_{N\to\infty}\int_{\Sb^3}f\left|P_{N,k}\right|^2\,d\Vol=\int_{\Sb^3}f\,d\Vol,$$
where $d\Vol$ is the Riemannian volume form on $\Sb^3$ normalized so that $\Vol(\Sb^3)=1$.
\end{thm}

\begin{proof}[Proof of Theorem \ref{thm:equid}]
Each $f\in C^\infty(\Sb^3)$ can be regarded as a semiclassical pseudo-differential symbol from $C^\infty(T^*\Sb^3)$ which is independent of the $\xi$-variable. Then $\Op_N(f)$ is the multiplication operator by $f$. Hence, by Theorem \ref{thm:sc},
\begin{eqnarray*}
\left\langle\Op_N(f)P_{N,k},P_{N,k}\right\rangle_{L^2(\Sb^3)}&=&\int_{\Sb^3}f\left|P_{N,k}\right|^2\,d\Vol\\
&\to&\int_0^1\int_{T_\rho}f(q)\,d\Area_\rho(q)d\rho\\
&=&\int_{\Sb^3}f\,d\Vol,
\end{eqnarray*}
proving Theorem \ref{thm:equid}.
\end{proof}

\subsection*{Background}
Our investigation of Bourgain's work on spherical harmonics \cite{B1} is inspired by the study of Laplacian eigenfunction behavior on manifolds, particularly concerning $L^p$ norm estimates and distribution, and how these properties are influenced by underlying geometry and the geodesic flow. For an overview of these topics, see Sogge \cite{So}.

On $\Sb^3$, certain spherical harmonics exhibit significant localization, with their $L^\infty$ norm growing at a specific rate with the degree. For instance, the spherical harmonic $z^j w^{N-j}$ with $0<j<N$, as considered in \eqref{eq:P}, is concentrated near the Clifford torus $T_\rho$ with $\rho = \frac{j}{N}$ and satisfies 
$$\frac{\left\|z^jw^{N-j}\right\|_{L^\infty(\Sb^3)}}{\left\|z^jw^{N-j}\right\|_{L^2(\Sb^3)}}\approx N^{\frac{1}{2}} \min\{j,N - j\}^{-\frac{1}{4}},$$
see Bourgain \cite[Equation (4)]{B1}.

In contrast, random spherical harmonics on $\Sb^3$ behave quite differently. Almost surely in a certain probabilistic sense, random spherical harmonics $u_N$ with degree $N$ equidistribute on $\Sb^3$ and satisfy 
$$\left\|u_N\right\|_{L^\infty(\Sb^3)}\approx\sqrt{\log N},\quad\text{where }\left\|u_N\right\|_{L^2(\Sb^3)}=1.$$
See Burq-Lebeau \cite{BL} and VanderKam \cite{V}. Furthermore, the semiclassical measure of $u_N$ is almost surely the (normalized) Liouville measure, that is, the uniform measure on the phase space $S^*\Sb^3$, see VanderKam \cite{V} and Zelditch \cite{Ze}. In this case, they are said to be quantum ergodic (the typical behavior of Laplacian eigenfunctions on manifolds with ergodic geodesic flow; see Zworski \cite[Chapter 15]{Zw}).

Bourgain’s spherical harmonics \eqref{eq:P} represent an exceptional case of uniformly bounded Laplacian eigenfunctions. Such eigenfunctions are rare across different geometries. Beyond the standard bases on Euclidean tori and rectangles, explicit constructions are known only on odd-dimensional spheres; see the works of Bourgain \cite{B1, B2}, Marzo-Ortega-Cerd\'a \cite{MOC}, and Shiffman \cite{Sh}. 

The existence of uniformly bounded spherical harmonics remains an open problem on even-dimensional spheres. On $\Sb^2$, the author \cite{H} constructed such spherical harmonics, under the assumption that there exist families of well-distributed points on the sphere. Here, ``well-distributed'' is quantified by a non-clustering condition around great circles, see Demeter-Zhang \cite{DZ} for further discussion. Moreover, it is also shown that these functions are quantum ergodic on $\Sb^2$ \cite{H}, which stands in sharp contrast to the behavior of the functions in \eqref{eq:P}. In particular,
our main result, Theorem \ref{thm:sc}, shows that Bourgain's spherical harmonics are localized on $\{T_\rho\times\{\xi_\rho\}:0\le\rho\le1\}\subset S^*\Sb^3$, and are not quantum ergodic.

Lastly, we mention that the Clifford torus $T_\frac12$ is an embedded minimal surface in $\Sb^3$. There is a close relation between the minimal surfaces and the first Laplacian eigenfunctions on them, for instance, Yau \cite[Problem 100]{Y} conjectured that the first (non-zero) Laplacian eigenvalue of an embedded minimal surface in $\Sb^n$ is $n-1$. Indeed, the first Laplacian eigenvalue of $T_\frac12$ is $2$ on $\Sb^3$, but Yau's conjecture remains open in full generality. However, our focus is on the Laplacian eigenfunctions on $\Sb^3$ in the high-eigenvalue limit (i.e., semiclassical limit), and the result of the relation between the Clifford tori in $\Sb^3$ and the semiclassical measure seems new. 

\section{Preliminaries}
In this section, we review the preliminaries on the Rudin-Shapiro sequences, geometry of $\Sb^3$ and the Clifford tori, and semiclassical pseudo-differential operators.

\subsection{Rudin-Shapiro sequences}\label{sec:RS}
Rudin and Shapiro \cite[Theorem 1]{R} constructed an example of a trigonometric series $P_N(t)$ of degree $N\in\N$ with coefficients $\pm1$ such that $\|P_N\|_{L^\infty}\approx\|P_N\|_{L^2}$ uniformly as $N\to\infty$.
\begin{defn}[Rudin-Shapiro polynomials and sequences]
Let $P_0=Q_0=1$. For $m\in\N$, inductively define
$$\begin{cases}
P_{m+1}(t)=P_m(t)+e^{i2^mt}Q_m(t),\\
Q_{m+1}(t)=P_m(t)-e^{i2^mt}Q_m(t).
\end{cases}$$
The resulting polynomials are written as
$$P_N(t)=\sum_{j=0}^N\sigma_j^Pe^{ijt}\quad\text{and}\quad Q_N(t)=\sum_{j=0}^N\sigma_j^Qe^{ijt}.$$
The sequences $\{\sigma_j^P\}_{j=1}^\infty$ and $\{\sigma_j^Q\}_{j=1}^\infty$ of $\pm1$ are called the Rudin-Shapiro sequences.
\end{defn}

A Rudin-Shapiro sequence has low auto-correlation. Indeed, Allouche-Choi-Denise-Erd\'elyi-Saffari \cite{ACDES} proved that
\begin{thm}\label{thm:auto}
There are absolute constants $C_0>0$ and $0<c_0<0.74$ such that
$$\left|\sum_{j=0}^N\sigma_j\sigma_{j+\beta}\right|\le C_0N^{c_0}\quad\text{for all }N\in\N\text{ and }\beta\in\Z\setminus\{0\}.$$
\end{thm}

\subsection{Hopf coordinates on $\Sb^3$}\label{sec:H}
The Hopf coordinates of a point $q=(x_1,y_1,x_2,y_2)\in\Sb^3$ are given by
$$x_1=\sqrt\rho\cos\theta_1,\quad y_1=\sqrt\rho\sin\theta_1,\quad x_2=\sqrt{1-\rho}\cos\theta_2,\quad 
y_2=\sqrt{1-\rho}\sin\theta_2,$$
where $\rho\in[0,1]$ and $\theta_1,\theta_2\in[0,2\pi)$. With
$$z=x_1+iy_1\quad\text{and}\quad w=x_2+iy_2,$$
we have that
$$(z,w)\in\C^2,\quad\text{where }\left|z\right|=\sqrt\rho\text{ and }\left|w\right|=\sqrt{1-\rho}.$$
The round metric at $q\in\Sb^3$ in the Hopf coordinates $(\rho,\theta_1,\theta_2)$ is given by
$$\left|\left(u,v_1,v_2\right)\right|_q^2=\frac{1}{4\rho(1-\rho)}|u|^2+\rho\left|v_1\right|^2+(1-\rho)\left|v_2\right|^2\quad\text{for }\left(u,v_1,v_2\right)\in T_p\Sb^3.$$
This induces a metric in the cotangent space $T_q^*\Sb^3$ via
\begin{equation}\label{eq:metric}
\left|\left(\eta,\xi_1,\xi_2\right)\right|_q^2=4\rho(1-\rho)|u|^2+\frac1\rho\left|\xi_1\right|^2+\frac{1}{1-\rho}\left|\xi_2\right|^2\quad\text{for }\left(\eta,\xi_1,\xi_2\right)\in T_q^*\Sb^3.
\end{equation}
\begin{ex}
Let $q=(\rho,\theta_1,\theta_2)\in\Sb^3$ in the Hopf coordinates. Define 
\begin{equation}\label{eq:xirho}
\xi_\rho=(0,\rho,1-\rho)\in T_q^*\Sb^3
\end{equation} 
Then by \eqref{eq:metric}, $|\xi_\rho|_q=1$, which indicates that $\xi_\rho\in S^*_q\Sb^3$, the cosphere space at $q\in\Sb^3$.
\end{ex}

The normalized Riemannian volume element $d\Vol$ on $\Sb^3$ in the Hopf coordinates $(\rho,\theta_1,\theta_2)$ is 
$$\,d\Vol=\frac{1}{4\pi^2}\,d\rho d\theta_1d\theta_2.$$
Therefore,
\begin{equation}\label{eq:norm}
\left\|z^jw^{N-j}\right\|_{L^2(\Sb^3)}^2=\frac{1}{4\pi^2}\int_0^{2\pi}\int_0^{2\pi}\int_0^1\rho^j(1-\rho)^{N-j}\,d\rho d\theta_1d\theta_2=\frac{j!(N-j)!}{(N+1)!},
\end{equation}
where we used the identity involving the Gamma function $\Gamma$ that
\begin{equation}\label{eq:beta}
\int_0^1\rho^j(1-\rho)^k\,d\rho=\frac{\Gamma(j+1)\Gamma(k+1)}{\Gamma(j+k+2)}\quad\text{for }j,k>-1.
\end{equation}
\begin{defn}[Clifford tori]
Let $0\le\rho\le1$. The Clifford torus $T_\rho$ is defined as
$$T_\rho=\left\{\left(\sqrt\rho\cos\theta_1,\sqrt\rho\sin\theta_1,\sqrt{1-\rho}\cos\theta_2,\sqrt{1-\rho}\sin\theta_2\right):0\le\theta_1,\theta_2<2\pi\right\},$$
which is equipped with the normalized area form
$$d\Area_\rho=\frac{1}{4\pi^2\sqrt{\rho(1-\rho)}}\,d\theta_1d\theta_2.$$
\end{defn}

\subsection{Semiclassical analysis}\label{sec:sc}
We collect the standard facts in semiclassical analysis that is sufficient for our purpose in this paper. We refer to Zworski \cite[Chapter 14]{Zw} for a systematic treatment on this topic.

For a spherical harmonic $u$ with eigenvalue $-N(N+2)$ on $\Sb^3$, the semiclassical parameter $N^{-1}\to0$ as $N\to\infty$. 

\begin{thm}[Semiclassical pseudo-differential operators]
There is a quantization rule $\Op_N$ which assigns to each $f\in C^\infty(T^*\Sb^3)$ with appropriate growth condition a family of semiclassical pseudo-differentialoperators $\Op_N(f)$. We call $f$ the symbol of $\Op_N(f)$.
\end{thm}

\begin{ex}
In the Hopf coordinates $q=(\rho,\theta_1,\theta_2)$, the functions 
$$\rho^\gamma e^{i\beta_1\theta_1}e^{i\beta_2\theta_2}\eta^a\xi_1^{b_1}\xi_2^{b_2},$$ 
where $(\eta,\xi_1,\xi_2)\in T^*_q\Sb^3$, $\gamma\in\N$, $\beta_1,\beta_2\in\Z$, and $a,b_1,b_2\in\N$, are symbols. (The growth rate in the fiber $(\eta,\xi_1,\xi_2)$ is polynomial.)
\end{ex}

\begin{rmk}
The quantization rule in the theorem above is not unique. For instance, let $f(q,\xi)\in C^\infty_0(T^*\Sb^3)$, where $q=(\rho,\theta_1,\theta_2)\in\Sb^3$ and $\xi=(\eta,\xi_1,\xi_2)\in T_q\Sb^3$ in the Hopf coordinates. In a local coordinate patch $\Omega\subset\Sb^3$, we can use the (left) quantization as
\begin{eqnarray*}
&&\Op_N(f)u\left(\rho,\theta_1,\theta_2\right)\\
&=&\left(\frac{N}{2\pi}\right)^3\int_{T^*\Sb^3}e^{iN\left[(\rho-\rho')\eta+\left(\theta_1-\theta_1'\right)\xi_1+\left(\theta_2-\theta_2'\right)\xi_2\right]}f\left(\rho,\theta_1,\theta_2,\eta,\xi_1,\xi_2\right)u\left(\rho',\theta_1',\theta_2'\right)\,d\rho'd\theta_1'd\theta_2'd\eta d\xi_1d\xi_2,
\end{eqnarray*}
where $u\in C^\infty(\Omega)$.
\end{rmk}

\begin{ex}
Let $f\in C^\infty(\Sb^3)$. Then 
$$\Op_N(f)u=fu+O_{L^2(\Sb^3)\to L^2(\Sb^3)}\left(N^{-1}\right).$$
\end{ex}

\begin{ex}
Let $f=\eta^a$ for $a\in\N$. Then 
$$\Op_N(f)u(q)=\left(\frac{\partial}{i\partial\rho}\right)^au\left(\rho,\theta_1,\theta_2\right)+O_{L^2(\Sb^3)\to L^2(\Sb^3)}\left(N^{-1}\right ).$$
\end{ex}

\begin{thm}\label{thm:sdo}\hfill
\begin{enumerate}[(i).]
\item If $f\in C^\infty_0(T^*\Sb^3)$, then
$$\left\|\Op_N(f)\right\|_{L^2(\Sb^3)\to L^2(\Sb^3)}=O(1),$$
where the implied constant depends on $C^m$ norms of $f$ for some absolute constant $m\in\N$ and is independent of $N$. 
\item If $f\in C^\infty_0(T^*\Sb^3)$, then
$$\Op_N(f)^\star=\Op_N\left(\ol f\right)+O_{L^2(\Sb^3)\to L^2(\Sb^3)}\left(N^{-1}\right),$$
where the implied constant depends on $C^m$ norms of $f$ for some absolute constant $m\in\N$.
\item If $f,g\in C^\infty_0(T^*\Sb^3)$, then
$$\Op_N(fg)=\Op_N(f)\Op_N(g)+O_{L^2(\Sb^3)\to L^2(\Sb^3)}\left(N^{-1}\right),$$
where the implied constant depends on $C^m$ norms of $f$ and $g$ for some absolute constant $m\in\N$.
\item Let $\Delta_{\Sb^3}u=-N(N+2)u$. Suppose that $\chi\in C^\infty(T^*\Sb^3)$ such that $\chi=1$ on a neighborhood of $S^*\Sb^3$. Then
$$\left\|\Op_N(\chi)u-u\right\|_{L^2(\Sb^3)}=O\left(N^{-\infty}\right).$$
\end{enumerate}
\end{thm}

Using the semiclassical pseudo-differential operators, we can define the semiclassical measure of spherical harmonics.

\begin{defn}[Semiclassical measures]
Let $\{u_j\}_{j=1}^\infty$ be a sequence of spherical harmonics on $\Sb^3$ with eigenvalues $-N_j(N_j+2)\to\infty$ as $j\to\infty$. These functions induce a semiclassical measure $\mu$ on $T^*\Sb^3$, if for each $f\in C_0^\infty(T^*\Sb^3)$,
$$\lim_{j\to\infty}\left\langle\Op_{N_j}(f)u_j,u_j\right\rangle_{L^2(\Sb^3)}=\int_{T^*\Sb^3}f\,d\mu.$$
\end{defn}

\section{Proof of Theorem \ref{thm:sc}}
Let $N\in\N$ and $k=0,...,N$. We prove Theorem \ref{thm:sc}, that is, for each $f\in C^\infty_0(T^*\Sb^3)$,
$$\lim_{N\to\infty}\left\langle\Op_N(f)P_{N,k},P_{N,k}\right\rangle_{L^2(\Sb^3)}=\int_0^1\int_{T_\rho}f\left(q,\xi_\rho\right)\,d\Area_\rho(q)d\rho.$$
Set $\chi\in C^\infty_0([\frac12,\frac32])$ such that $\chi(\xi)=1$ if $\frac34\le|\xi|_q\le\frac54$. By Theorem \ref{thm:sdo},
$$\left\|\Op_N(f)P_{N,k}-\Op_N(f\chi)P_{N,k}\right\|_{L^2(\Sb^3)}=O\left(N^{-1}\right).$$
To identify the semiclassical measure of $P_{N,k}$, it therefore suffices to consider the symbols $f\chi$, where $f\in C_0^\infty(T^*\Sb^3)$. In the family of such functions, the ones of the form
$$f\left(\rho,\theta_1,\theta_2,\eta,\xi_1,\xi_2\right)=\rho^\gamma e^{i\beta_1\theta_1}e^{i\beta_2\theta_2}\eta^a\xi_1^{b_1}\xi_2^{b_2}\chi\left(\rho,\theta_1,\theta_2,\eta,\xi_1,\xi_2\right),$$ 
where $(\eta,\xi_1,\xi_2)\in T^*_q\Sb^3$, $\gamma\in\N$, $\beta_1,\beta_2\in\Z$, and $a,b_1,b_2\in\N$, are dense in $C_0^\infty(T^*\Sb^3)$. Hence, it further suffices to consider the semiclassical pseudo-differential operators with symbols of the form above, c.f., Zworski \cite[Theorem 5.2]{Zw}.

Next, we make another reduction at the two ends of summation for $P_{N,k}$ when considering their semiclassical measure. Recall that
$$P_{N,k}(z,w)=\frac{1}{\sqrt{N+1}}\sum_{j=0}^N\sigma_je^{\frac{2\pi ijk}{N+1}}\frac{z^jw^{N-j}}{\left\|z^jw^{N-j}\right\|_{L^2(\Sb^3)}}.$$
Let $0<\ve<\frac12$. Define
\begin{equation}\label{eq:Pve}
P_\ve(z,w)=\frac{1}{\sqrt{N+1}}\sum_{N^\ve\le j\le N-N^\ve}\sigma_je^{\frac{2\pi ijk}{N+1}}\frac{z^jw^{N-j}}{\left\|z^jw^{N-j}\right\|_{L^2(\Sb^3)}}.
\end{equation}
Notice that $|\sigma_j|=1$ for all $j\in\N$. Then
\begin{eqnarray*}
\left\|P_{N,k}-P_\ve\right\|_{L^2(\Sb^3)}&=&\frac{1}{\sqrt{N+1}}\left\|\left(\sum_{0\le j<N^\ve}+\sum_{N-N^\ve<j\le N}\right)\sigma_je^{\frac{2\pi ijk}{N+1}}\frac{z^jw^{N-j}}{\left\|z^jw^{N-j}\right\|_{L^2(\Sb^3)}}\right\|_{L^2(\Sb^3)}\\
&\le&2N^{\ve-\frac12}.
\end{eqnarray*}
Since $\|P_{N,k}\|_{L^2(\Sb^3)}=1$ and $\|P_\ve\|_{L^2(\Sb^3)}<1$,
\begin{eqnarray*}
&&\left|\left\langle\Op_N(f)P_{N,k},P_{N,k}\right\rangle_{L^2(\Sb^3)}-\left\langle\Op_N(f)P_\ve,P_\ve\right\rangle_{L^2(\Sb^3)}\right|\\
&=&\left|\left\langle\Op_N(f)P_{N,k},P_{N,k}-P_\ve\right\rangle_{L^2(\Sb^3)}+\left\langle\Op_N(f)\left(P_{N,k}-P_\ve\right),P_\ve\right\rangle_{L^2(\Sb^3)}\right|\\
&\le&2\left\|\Op_N(f)\right\|_{L^2(\Sb^3)\to L^2(\Sb^3)}\left\|P_{N,k}-P_\ve\right\|_{L^2(\Sb^3)}\\
&=&O_f\left(N^{\ve-\frac12}\right)\to0\quad\text{as }N\to\infty.
\end{eqnarray*}
Therefore, to prove Theorem \ref{thm:sc}, it suffices to consider $P_\ve$ instead of $P_{N,k}$:
$$\lim_{N\to\infty}\left\langle\Op_N(f)P_\ve,P_\ve\right\rangle_{L^2(\Sb^3)}=\int_0^1\int_{T_\rho}f\left(q,\xi_\rho\right)\,d\Area_\rho(q)d\rho,$$
where 
\begin{itemize}
\item $f(\rho,\theta_1,\theta_2,\eta,\xi_1,\xi_2)=\rho^\gamma e^{i\beta_1\theta_1}e^{i\beta_2\theta_2}\eta^a\xi_1^{b_1}\xi_2^{b_2}\chi(\rho,\theta_1,\theta_2,\eta,\xi_1,\xi_2)$, 
\item $j\to\infty$ and $N-j\to\infty$ as $N\to\infty$ in the summation for $P_\ve$ in \eqref{eq:Pve}.
\end{itemize}
On one hand, with $\xi_\rho=(0,\rho,1-\rho)$ in \eqref{eq:xirho}, 
\begin{equation}\label{eq:0}
\int_0^1\int_{T_\rho}f\left(q,\xi_\rho\right)\,d\Area_\rho(q)d\rho=\delta_0(a)\delta_0\left(\beta_1\right)\delta_0\left(\beta_2\right)\int_0^1\rho^{\gamma+b_1}(1-\rho)^{b_2}\,d\rho.
\end{equation}
On the other hand, let $f_1=\rho^\gamma e^{i\beta_1\theta_1}e^{i\beta_2\theta_2}$ and $f_2=\eta^a\xi_1^{b_1}\xi_2^{b_2}$.  By Theorem \ref{thm:sdo},
\begin{eqnarray*}
&&\left\langle\Op_N(f)P_\ve,P_\ve\right\rangle_{L^2(\Sb^3)}\\&=&\left\langle\Op_N\left(\chi f_1\chi f_2\right)P_\ve,P_\ve\right\rangle_{L^2(\Sb^3)}+O\left(N^{-1}\right)\\
&=&\left\langle\Op_N\left(\chi f_1\right)\Op_N\left(\chi f_2\right)P_\ve,P_\ve\right\rangle_{L^2(\Sb^3)}+O\left(N^{-1}\right)\\
&=&\left\langle\Op_N\left(\chi f_2\right)P_\ve,\Op_N\left(\ol{\chi f_1}\right)P_\ve\right\rangle_{L^2(\Sb^3)}+O\left(N^{-1}\right)\\
&=&\left\langle\Op_N\left(f_2\right)P_\ve,\Op_N\left(\ol{f_1}\right)P_\ve\right\rangle_{L^2(\Sb^3)}+O\left(N^{-1}\right)\\
&=&\left\langle\Op_N\left(\eta^a\xi_1^{b_1}\xi_2^{b_2}\right)P_\ve,\rho^\gamma e^{-i\beta_1\theta_1}e^{-i\beta_2\theta_2}P_\ve\right\rangle_{L^2(\Sb^3)}+O\left(N^{-1}\right).
\end{eqnarray*}
Here,
\begin{eqnarray*}
&&\Op_N\left(\eta^a\xi_1^{b_1}\xi_2^{b_2}\right)P_\ve\\
&=&\frac{1}{(iN)^{a+b_1+b_2}\sqrt{N+1}}\sum_{j=N^\ve}^{N-N^\ve}\frac{\sigma_je^{\frac{2\pi ikj}{N+1}}}{\left\|z^jw^{N-j}\right\|_{L^2(\Sb^3)}}\left(\frac{\partial}{\partial\rho}\right)^a\left(\frac{\partial}{\partial\theta_1}\right)^{b_1}\left(\frac{\partial}{\partial\theta_2}\right)^{b_2}\rho^\frac j2(1-\rho)^{\frac{N-j}{2}}e^{ij\theta_1}e^{i(N-j)\theta_2}\\
&=&\frac{1}{i^aN^{a+b_1+b_2}\sqrt{N+1}}\sum_{j=N^\ve}^{N-N^\ve}\frac{\sigma_je^{\frac{2\pi ikj}{N+1}}j^{b_1}(N-j)^{b_2}}{\left\|z^jw^{N-j}\right\|_{L^2(\Sb^3)}}\left(\frac{\partial}{\partial\rho}\right)^a\left[\rho^\frac j2(1-\rho)^{\frac{N-j}{2}}\right]e^{ij\theta_1}e^{i(N-j)\theta_2},
\end{eqnarray*}
and
$$\rho^\gamma e^{-i\beta_1\theta_1}e^{-i\beta_2\theta_2}P_\ve=\frac{1}{\sqrt{N+1}}\sum_{l=N^\ve}^{N-N^\ve}\frac{\sigma_le^{\frac{2\pi ikl}{N+1}}}{\left\|z^lw^{N-l}\right\|_{L^2(\Sb^3)}}\rho^{\frac l2+\gamma}(1-\rho)^{\frac{N-l}{2}}e^{i\left(l-\beta_1\right)\theta_1}e^{i\left(N-l-\beta_2\right)\theta_2}.$$
By \eqref{eq:norm},
\begin{eqnarray*}
&&\left\langle\Op_N(f)P_\ve,P_\ve\right\rangle_{L^2(\Sb^3)}\\
&=&\frac{N!}{4\pi^2i^aN^a}\sum_{j,l=N^\ve}^{N-N^\ve}\sigma_j\sigma_le^{\frac{2\pi ik(j-l)}{N+1}}\cdot\frac{1}{\sqrt{j!(N-j)!}\sqrt{l!(N-l)!}}\cdot\left(\frac jN\right)^{b_1}\left(1-\frac jN\right)^{b_2}\\
&&\cdot\int_0^{2\pi}\int_0^{2\pi}\int_0^1e^{i\left(j-l+\beta_1\right)\theta_1}e^{i\left(l-j+\beta_2\right)\theta_2}\left(\frac{d}{d\rho}\right)^a\left[\rho^\frac j2(1-\rho)^{\frac{N-j}{2}}\right]\rho^{\frac l2+\gamma}(1-\rho)^{\frac{N-l}{2}}\,d\rho d\theta_1d\theta_2.
\end{eqnarray*}
The terms in the summation are non-zero only if $\beta_1=-\beta_2=\beta$ for some $\beta\in\Z$ and $l=j+\beta$. Under this condition, the above equation continues:
\begin{eqnarray}
&&\frac{e^{-\frac{2\pi ik\beta}{N+1}}N!}{i^aN^a}\sum_{j=N^\ve}^{N-N^\ve}\sigma_j\sigma_{j+\beta}\cdot\frac{1}{\sqrt{j!(N-j)!}\sqrt{(j+\beta)!(N-j-\beta)!}}\cdot\left(\frac jN\right)^{b_1}\left(1-\frac jN\right)^{b_2}\nonumber\\
&&\cdot\int_0^1\left(\frac{d}{d\rho}\right)^a\left[\rho^\frac j2(1-\rho)^{\frac{N-j}{2}}\right]\rho^{\frac j2+\frac\beta2+\gamma}(1-\rho)^{\frac{N-j}{2}-\frac\beta2}\,d\rho.\label{eq:sum}
\end{eqnarray}
To prove Theorem \ref{thm:sc}, we need to show that \eqref{eq:sum} $\to$ \eqref{eq:0} as $N\to\infty$. We divide our analysis in the following three cases.
\begin{itemize}
\item Case 1. $a=0$ and $\beta=0$.
\item Case 2. $a=0$ and $\beta\in\Z\setminus\{0\}$.
\item Case 3. $a\in\N\setminus\{0\}$.
\end{itemize}

\subsection{Case 1} 
Suppose that $a=0$ and $\beta=0$. By \eqref{eq:beta},
\begin{eqnarray*}
\eqref{eq:sum}&=&N!\sum_{j=N^\ve}^{N-N^\ve}\left|\sigma_j\right|^2\cdot\frac{1}{j!(N-j)!}\cdot\left(\frac jN\right)^{b_1}\left(1-\frac jN\right)^{b_2}\cdot\int_0^1\rho^{j+\gamma}(1-\rho)^{N-j}\,d\rho\\
&=&N!\sum_{j=N^\ve}^{N-N^\ve}\frac{(j+\gamma)!}{j!(N+\gamma+1)!}\cdot\left(\frac jN\right)^{b_1}\left(1-\frac jN\right)^{b_2}\\
&=&\frac{1}{(N+1)\cdots(N+\gamma+1)}\sum_{j=N^\ve}^{N-N^\ve}(j+1)\cdots(j+\gamma)\left(\frac jN\right)^{b_1}\left(1-\frac jN\right)^{b_2}\\
&=&\frac{1}{(N+1)\cdots(N+\gamma+1)}\sum_{j=N^\ve}^{N-N^\ve}\left(j^\gamma+O\left(j^{\gamma-1}\right)\right)\left(\frac jN\right)^{b_1}\left(1-\frac jN\right)^{b_2}\\
&=&\frac{N^{\gamma+1}}{(N+1)\cdots(N+\gamma+1)}\sum_{j=N^\ve}^{N-N^\ve}\left(\frac jN\right)^{\gamma+b_1}\left(1-\frac jN\right)^{b_2}\frac1N+O\left(N^{-1}\right).
\end{eqnarray*}
Notice that as $N\to\infty$,
$$\sum_{j=0}^{N^\ve-1}\left(\frac jN\right)^{\gamma+b_1}\left(1-\frac jN\right)^{b_2}\frac1N=O\left(N^{\ve-1}\right)$$
and
$$\sum_{j=N-N^\ve+1}^{N}\left(\frac jN\right)^{\gamma+b_1}\left(1-\frac jN\right)^{b_2}\frac1N=O\left(N^{\ve-1}\right),$$
whereas
$$\sum_{j=0}^{N}\left(\frac jN\right)^{\gamma+b_1}\left(1-\frac jN\right)^{b_2}\frac1N\to\int_0^1\rho^{\gamma+b_1}(1-\rho)^{b_2}\,d\rho.$$
Hence,
\begin{eqnarray*}
\eqref{eq:sum}&=&\frac{N^{\gamma+1}}{(N+1)\cdots(N+\gamma+1)}\sum_{j=N^\ve}^{N-N^\ve}\left(\frac jN\right)^{\gamma+b_1}\left(1-\frac jN\right)^{b_2}\frac1N+O\left(N^{-1}\right)\\
&\to&\int_0^1\rho^{\gamma+b_1}(1-\rho)^{b_2}\,d\rho\quad\text{as }N\to\infty,
\end{eqnarray*}
which agrees with \eqref{eq:0} in this case.

\subsection{Case 2} 
Suppose that $a=0$ and $\beta\in\Z\setminus\{0\}$. By \eqref{eq:beta},
\begin{eqnarray*}
\eqref{eq:sum}&=&e^{-\frac{2\pi ik\beta}{N+1}}N!\sum_{j=N^\ve}^{N-N^\ve}\sigma_j\sigma_{j+\beta}\cdot\frac{1}{\sqrt{j!(N-j)!}\sqrt{(j+\beta)!(N-j-\beta)!}}\cdot\left(\frac jN\right)^{b_1}\left(1-\frac jN\right)^{b_2}\\
&&\cdot\int_0^1\rho^{j+\gamma+\frac\beta2}(1-\rho)^{N-j-\frac\beta2}\,d\rho\\
&=&\frac{e^{-\frac{2\pi ik\beta}{N+1}}N!}{(N+\gamma+1)!}\sum_{j=N^\ve}^{N-N^\ve}\sigma_j\sigma_{j+\beta}\cdot\frac{\Gamma\left(j+\gamma+\frac\beta2+1\right)\Gamma\left(N-j-\frac\beta2+1\right)}{\sqrt{j!(j+\beta)!}\sqrt{(N-j-\beta)!(N-j)!}}\cdot\left(\frac jN\right)^{b_1}\left(1-\frac jN\right)^{b_2}
\end{eqnarray*}
Denote by 
$$A_j=\sigma_j\sigma_{j+\beta}\quad{and}\quad B_j=\frac{\Gamma\left(j+\gamma+\frac\beta2+1\right)\Gamma\left(N-j-\frac\beta2+1\right)}{\sqrt{j!(j+\beta)!}\sqrt{(N-j-\beta)!(N-j)!}}\cdot\left(\frac jN\right)^{b_1}\left(1-\frac jN\right)^{b_2}.$$
By Theorem \ref{thm:auto}, 
$$\sum_{l=0}^jA_l=O\left(j^{c_0}\right).$$
By the Stirling formula for the Gamma functions,
$$\frac{\Gamma\left(j+\gamma+\frac\beta2+1\right)\Gamma\left(N-j-\frac\beta2+1\right)}{\sqrt{j!(j+\beta)!}\sqrt{(N-j-\beta)!(N-j)!}}=j^\gamma\left(1+O\left(j^{-1}\right)+O\left((N-j)^{-1}\right)\right).$$
It then follows that
$$B_{j+1}-B_j=O\left(j^{\gamma-1}\right).$$
Applying Abel's summation by parts, 
\begin{eqnarray*}
\sum_{j=N^\ve}^{N-N^\ve}A_jB_j&=&B_N\sum_{j=N^\ve}^{N-N^\ve}A_j-\sum_{j=0}^{N-1}\left(\sum_{l=0}^jA_l\right)\left(B_{j+1}-B_j\right)\\
&=&O\left(N^{\gamma+c_0}\right)-\sum_{j=0}^{N-1}O\left(j^{c_0+\gamma-1}\right)\\
&=&O\left(N^{\gamma+c_0}\right).
\end{eqnarray*}
Therefore,
$$\eqref{eq:sum}=\frac{e^{-\frac{2\pi ik\beta}{N+1}}N!}{(N+\gamma+1)!}\sum_{j=N^\ve}^{N-N^\ve}A_jB_j=O\left(N^{c_0-1}\right)\to0\quad\text{as }N\to\infty,$$
because $0<c_0<0.74$ in Theorem \ref{thm:auto}. 

\subsection{Case 3} 
Suppose that $a\in\N\setminus\{0\}$. Then
\begin{eqnarray*}
\eqref{eq:sum}&=&\frac{e^{-\frac{2\pi ik\beta}{N+1}}N!}{i^aN^a}\sum_{j=N^\ve}^{N-N^\ve}\sigma_j\sigma_{j+\beta}\cdot\frac{1}{\sqrt{j!(N-j)!}\sqrt{(j+\beta)!(N-j-\beta)!}}\cdot\left(\frac jN\right)^{b_1}\left(1-\frac jN\right)^{b_2}\\
&&\cdot\int_0^1\left(\frac{d}{d\rho}\right)^a\left[\rho^\frac j2(1-\rho)^{\frac{N-j}{2}}\right]\rho^{\frac j2+\gamma+\frac\beta2}(1-\rho)^{\frac{N-j}{2}-\frac\beta2}\,d\rho,
\end{eqnarray*}
where $\gamma,b_1,b_2$ are fixed. We begin by estimating the following factor, which plays a critical role in the overall estimate of the sum above.
\begin{lemma}
Let $N^\ve\le j\le N-N^\ve$. Then
\begin{eqnarray*}
&&\left|\frac{1}{\sqrt{j!(N-j)!}\sqrt{(j+\beta)!(N-j-\beta)!}}\int_0^1\left(\frac{d}{d\rho}\right)^a\left[\rho^\frac j2(1-\rho)^{\frac{N-j}{2}}\right]\rho^{\frac j2+\gamma+\frac\beta2}(1-\rho)^{\frac{N-j}{2}-\frac\beta2}\,d\rho\right|\\
&\le&\frac{Cj^\gamma}{(N+\gamma-a+1)!N},
\end{eqnarray*}
where $C=C(\gamma,\beta,a)>0$.
\end{lemma}
\begin{proof}
With $a=a_1+a_2$ such that $a_1,a_2\in\N$, 
\begin{eqnarray*}
&&\frac{1}{\sqrt{j!(N-j)!}\sqrt{(j+\beta)!(N-j-\beta)!}}\int_0^1\left(\frac{d}{d\rho}\right)^a\left[\rho^\frac j2(1-\rho)^{\frac{N-j}{2}}\right]\rho^{\frac j2+\gamma+\frac\beta2}(1-\rho)^{\frac{N-j}{2}-\frac\beta2}\,d\rho\\
&=&\frac{1}{\sqrt{j!(N-j)!}\sqrt{(j+\beta)!(N-j-\beta)!}}\sum_{a_1+a_2=a}{a\choose a_1}(-1)^{a_2}\left(\frac j2\right)\left(\frac j2-1\right)\cdots\left(\frac j2-a_1+1\right)\\
&&\cdot\left(\frac{N-j}{2}\right)\left(\frac{N-j}{2}-1\right)\cdots\left(\frac{N-j}{2}-a_2+1\right)\\
&&\cdot\frac{\Gamma\left(j+\gamma+\frac\beta2-a_1+1\right)\Gamma\left(N-j-\frac\beta2-a_2+1\right)}{(N+\gamma-a+1)!}\\
&=&\frac{1}{2^a\sqrt{j!(N-j)!}\sqrt{(j+\beta)!(N-j-\beta)!}(N+\gamma-a+1)!}\\
&&\cdot\sum_{a_1+a_2=a}{a\choose a_1}(-1)^{a_2}j(j-2)\cdots\left(j-2a_1+2\right)\\
&&\cdot(N-j)(N-j-2)\cdots\left(N-j-2a_2+2\right)\cdot\Gamma\left(j+\gamma+\frac\beta2-a_1+1\right)\Gamma\left(N-j-\frac\beta2-a_2+1\right)\\
&=&\frac{1}{2^a(N+\gamma-a+1)!}\sum_{a_1+a_2=a}{a\choose a_1}(-1)^{a_2}j(j-2)\cdots\left(j-2a_1+2\right)\\
&&\cdot(N-j)(N-j-2)\cdots\left(N-j-2a_2+2\right)\cdot\frac{\Gamma\left(j+\gamma+\frac\beta2-a_1+1\right)}{\sqrt{j!(j+\beta)!}}\frac{\Gamma\left(N-j-\frac\beta2-a_2+1\right)}{\sqrt{(N-j)!(N-j-\beta)!}}.
\end{eqnarray*}
Notice that 
\begin{eqnarray*}
&&\sum_{a_1+a_2=a}{a\choose a_1}(-1)^{a_2}j(j-1)\cdots\left(j-a_1+1\right)\\
&&\cdot(N-j)(N-j-1)\cdots\left(N-j-a_2+1\right)\cdot\frac{j^\gamma\left(j-a_1\right)!}{j!}\frac{\left(N-j-a_2\right)!}{(N-j)!}\\
&=&j^\gamma\sum_{a_1+a_2=a}{a\choose a_1}(-1)^{a_2}\\
&=&0.
\end{eqnarray*}
Moreover,
$$\frac{\Gamma\left(j+\gamma+\frac\beta2-a_1+1\right)}{\sqrt{j!(j+\beta)!}}=\left[1+O\left(j^{-1}\right)\right]\frac{j^\gamma\left(j-a_1\right)!}{j!},$$
and
$$\frac{\Gamma\left(N-j-\frac\beta2-a_2+1\right)}{\sqrt{(N-j)!(N-j-\beta)!}}=\left[1+O\left((N-j)^{-1}\right)\right]\frac{\left(N-j-a_2\right)!}{(N-j)!}.$$
It the follows that
\begin{eqnarray*}
&&j(j-2)\cdots\left(j-2a_1+2\right)\cdot(N-j)(N-j-2)\cdots\left(N-j-2a_2+2\right)\\
&&-j(j-1)\cdots\left(j-a_1+1\right)\cdot(N-j)(N-j-1)\cdots\left(N-j-a_2+1\right)\\
&=&O\left(j^{a_1-1}(N-j)^{a_2}+j^{a_1}(N-j)^{a_2-1}\right).
\end{eqnarray*}
Hence,
\begin{eqnarray*}
&&\frac{1}{2^a(N+\gamma-a+1)!}\sum_{a_1+a_2=a}{a\choose a_1}(-1)^{a_2}j(j-2)\cdots\left(j-2a_1+2\right)\\
&&\cdot(N-j)(N-j-2)\cdots\left(N-j-2a_2+2\right)\cdot\frac{\Gamma\left(j+\gamma+\frac\beta2-a_1+1\right)}{\sqrt{j!(j+\beta)!}}\frac{\Gamma\left(N-j-\frac\beta2-a_2+1\right)}{\sqrt{(N-j)!(N-j-\beta)!}}\\
&=&\frac{1}{2^a(N+\gamma-a+1)!}\sum_{a_1+a_2=a}{a\choose a_1}(-1)^{a_2}\Big[j(j-1)\cdots\left(j-a_1+1\right)\\
&&\cdot(N-j)(N-j-1)\cdots\left(N-j-a_2+1\right)+O\left(j^{a_1-1}(N-j)^{a_2}+j^{a_1}(N-j)^{a_2-1}\right)\Big]\\&&\cdot\left[1+O\left(j^{-1}\right)+O\left((N-j)^{-1}\right)\right]\frac{j^\gamma\left(j-a_1\right)!}{j!}\frac{\left(N-j-a_2\right)!}{(N-j)!}\\
&=&\frac{j^\gamma}{(N+\gamma-a+1)!}\cdot\left[O\left(j^{-1}\right)+O\left((N-j)^{-1}\right)\right]\\
&&+\frac{1}{(N+\gamma-a+1)!}\cdot O\left(j^{a_1-1}(N-j)^{a_2}+j^{a_1}(N-j)^{a_2-1}\right)\cdot\frac{j^\gamma\left(j-a_1\right)!}{j!}\frac{\left(N-j-a_2\right)!}{(N-j)!}\\
&=&\frac{j^\gamma}{(N+\gamma-a+1)!}\cdot O\left(N^{-1}\right).
\end{eqnarray*}
That is,
\begin{eqnarray*}
&&\left|\frac{1}{\sqrt{j!(N-j)!}\sqrt{(j+\beta)!(N-j-\beta)!}}\int_0^1\left(\frac{d}{d\rho}\right)^a\left[\rho^\frac j2(1-\rho)^{\frac{N-j}{2}}\right]\rho^{\frac j2+\gamma+\frac\beta2}(1-\rho)^{\frac{N-j}{2}-\frac\beta2}\,d\rho\right|\\
&\le&\frac{Cj^\gamma}{(N+\gamma-a+1)!N}.
\end{eqnarray*}
\end{proof}
Following the lemma, we have that
\begin{eqnarray*}
&&\left|\eqref{eq:sum}\right|\\
&\le&\frac{N!}{N^a}\sum_{j=N^\ve}^{N-N^\ve}\left(\frac jN\right)^{b_1}\left(1-\frac jN\right)^{b_2}\\
&&\cdot\left|\frac{1}{\sqrt{j!(N-j)!}\sqrt{(j+\beta)!(N-j-\beta)!}}\int_0^1\left(\frac{d}{d\rho}\right)^a\left[\rho^\frac j2(1-\rho)^{\frac{N-j}{2}}\right]\rho^{\frac j2+\gamma+\frac\beta2}(1-\rho)^{\frac{N-j}{2}-\frac\beta2}\,d\rho\right|\\
&\le&\frac{N!}{N^a}\sum_{j=N^\ve}^{N-N^\ve}\left(\frac jN\right)^{b_1}\left(1-\frac jN\right)^{b_2}\frac{Cj^\gamma}{(N+\gamma-a+1)!N}\\
&\le&\frac{CN^{-\gamma+a-1}}{N^a}\sum_{j=N^\ve}^{N-N^\ve}\left(\frac jN\right)^{b_1}\left(1-\frac jN\right)^{b_2}\frac{j^\gamma}{N}\\
&\le&CN^{-1}\sum_{j=N^\ve}^{N-N^\ve}\left(\frac jN\right)^{\gamma+b_1}\left(1-\frac jN\right)^{b_2}\frac{1}{N}\\
&\le&CN^{-1}\int_0^1\rho^{\gamma+b_1}(1-\rho)^{b_2}\,d\rho\\
&\to&0\quad\text{as }N\to\infty.
\end{eqnarray*}
With the three cases complete, we conclude the proof of Theorem \ref{thm:sc}.

\section*{Acknowledgements}
I would like to thank the reviewer for the numerous suggestions, which have significantly improved the presentation of the article.

\end{document}